\documentclass[11pt,twoside]{amsart}
\usepackage[all]{xy}   
        \usepackage {amssymb,latexsym,amsthm,amsmath}
        
\long\def\symbolfootnote[#1]#2{\begingroup%
\def\thefootnote{\fnsymbol{footnote}}\footnote[#1]{#2}\endgroup}

\newcommand{\diag}{\textup{diag}}

\newcommand{\cross}{\times}
\newcommand{\tensor}{\otimes}
\makeatletter
\def\imod#1{\allowbreak\mkern10mu({\operator@font mod}\,\,#1)}
\makeatother

\newtheorem{theorem}{Theorem}[section]
\newtheorem{lemma}[theorem]{Lemma}
\newtheorem{corollary}[theorem]{Corollary}
\newtheorem{proposition}[theorem]{Proposition}
\newtheorem*{theorem*}{Theorem}
\theoremstyle{definition}
\newtheorem{remark}[theorem]{Remark}
\numberwithin{equation}{section}

\newcommand{\ignore}[1]{}

\newcommand{\mynote}[1]{}
\begin{document}
\setcounter{section}{0}
\title{Real Elements in Spin Groups}
\author{Anupam Singh}
\address{The Institute of Mathematical Sciences\\ C.I.T. Campus Taramani, Chennai 600113 INDIA}
\email{anupamk18@gmail.com}
\date{}

\begin{abstract}
Let $F$ be a field of characteristic $\neq 2$.
Let $G$ be an algebraic group defined over $F$.
An element $t\in G(F)$ is called {\bf real} if there exists $s\in G(F)$ such that $sts^{-1}=t^{-1}$.
A semisimple element $t$ in $GL_n(F), SL_n(F), O(q), SO(q), Sp(2n)$ and the groups of type $G_2$ over $F$ is real if and only if $t=\tau_1\tau_2$ where $\tau_1^2=\pm 1=\tau_2^2$ (ref. \cite{st1,st2}).
In this paper we extend this result to the semisimple elements in $Spin$ groups when $\dim(V)\equiv 0,1,2 \imod 4$. 
\end{abstract}
\maketitle
{\it Keywords: Spin Groups, Clifford Groups, Real Elements etc.}
\section{Introduction}

Let $F$ be a field of characteristic $\neq 2$.
Let $G$ be a linear algebraic group defined over $F$.
An element $t\in G(F)$ is called {\bf real} if there exists $s\in G(F)$ such that $sts^{-1}=t^{-1}$.
In~\cite{st1} and~\cite{st2}, we studied real elements in classical groups and the groups of type $G_2$ and proved following result,
\begin{theorem}\label{st12}
Let $t$ be an element of $G$ of one of the following type: any element in $GL_n(F), SL_n(F)$ or groups of type $G_2$ defined over $F$ or a semisimple element in groups of type $A_1$, $O(q)$, $SO(q)$, or $Sp(2n)$.
Then $t$ is real in $G$ if and only if $t$ has a decomposition $t=\tau_1\tau_2$ where $\tau_1, \tau_2\in G$ and $\tau_1^2=\pm 1=\tau_2^2$.
\end{theorem}
\noindent Notice that if $t\in G$ is real then $t$ has decomposition $t=\tau_1\tau_2$ with the property in the theorem if and only if there exists $\tau\in G$ such that $\tau t\tau^{-1}=t^{-1}$ and $\tau^2=\pm 1$.
In this paper, we study reality of semisimple elements in $Spin$ groups.  
In ~\cite{fz}, Feit and Zuckerman studied reality of elements in $Spin$ groups (for definitions of these groups see next section).
But they consider the conjugating element coming from the bigger group, namely, the group extended by Dynkin diagram automorphisms.
Let $F$ be an algebraically closed field or a finite field.
They prove (Theorem B, in~\cite{fz}),
\begin{theorem}
Any element of $Spin(V,q)$ is real in $\Gamma(V,q)$.
\end{theorem} 
\noindent From this theorem they deduce that if $\dim(V)$ is odd and $F$ is algebraically closed then every element of $Spin$ is real in $Spin$.
Since $Spin_3\cong SL_2$, one can produce unipotent elements which are not always real.
The results for $Spin$ groups similar to the results in Theorem~\ref{st12} are known over algebraically closed fields and finite fields. 
This we have collected in Proposition~\ref{realalgclosed} and~\ref{realfinitefield}.
In~\cite{tz} Tiep and Zalesski study the real conjugacy classes in groups of Lie type where they consider the $Spin$ groups also.
They produce examples of non-real elements (Section 6) but most of them are not semisimple.
For the general results about real elements in algebraic groups see~\cite{st2} Section 2.1 and 2.3 and Theorem 1.5 in \cite{tz}.

Let $F$ be a perfect field of characteristic $\neq 2$.
Let $V$ be a vector space with a quadratic form $q$ on it.
In this paper, we prove that all semisimple elements in $Spin(V,q)$ are real (over $F$) if $\dim(V)\equiv 0$ or $1 \imod 4$ (see Corollary~\ref{real01spin}).
We also prove that if  $\dim(V)\equiv 0,1$ or $2 \imod 4$ then every real semisimple element $t$ of $Spin(V,q)$ has a decomposition $t=\tau_1\tau_2$ where $\tau_1^2=\pm 1=\tau_2^2$ (Theorem~\ref{realspin}).
We (Corollary~\ref{counterexample}) also give examples when elements are not product of two involutions but they are product of elements of which square is $-1$.
Notice that the reality results for $SO(V,q)$ over $F$ does not lift to $Spin(V,q)$ (as it does over algebraically closed field, Proposition~\ref{realalgclosed}) since the image of $Spin(V,q)$ in $SO(V,q)$ is the subgroup $O'(V,q)$.
We first prove that for any semisimple element $t\in \Gamma^+(V,q)$, the even Clifford group, there exists $s\in \Gamma(V,q)$ such that $sts^{-1}=N(t)t^{-1}$ with $N(s)=1$ and $s^2=\pm 1$.
The investigation proceeds to find such conjugating element in the group $Spin(V,q)$.

\section{Clifford Groups and Spin Groups}

In this section we fix the notation and terminology (see~\cite{j} Section 4.8).
Let $V$ be an $n$-dimensional vector space over field $F$ with a quadratic form $q$.
We denote by $B$ the corresponding bilinear form.
Let 
$$
C(V,q)=\frac{T(V)}{\langle x\tensor x-q(x).1 \mid x\in V\rangle}
$$
be the {\bf Clifford algebra} where $T(V)$ is the tensor algebra.
Then $C(V,q)$ is $\mathbb Z/2\mathbb Z$-graded algebra, say, $C(V,q)=C_0(V,q)+C_1(V,q)$.
The subalgebra $C_0(V,q)$ is called {\bf special (or even) Clifford algebra} and it is a Clifford algebra in its own right.
The group $\Gamma(V,q)=\{u\in C(V,q)^{\cross}\mid uxu^{-1}\in V \ \forall x\in V\}$ is called {\bf Clifford group} and $\Gamma^+(V,q)=\Gamma(V,q)\cap C_0(V,q)$ is called {\bf special Clifford group}.
 
We have a representation for Clifford group, called vector representation, $\chi\colon \Gamma(V,q)\rightarrow O(V,q)$ defined by $v\mapsto -S_v$ for $v\in V\subset C(V,q)$ with $q(v)\neq 0$. 
Here $S_v$ is a reflection defined by $S_v(x)=x-\frac{B(v,x)}{q(v)}v$.
This gives rise to the following exact sequence:
$$1\rightarrow F^*\rightarrow \Gamma^+(V,q)\stackrel{\chi}{\rightarrow} SO(q)\rightarrow 1$$
and with this representation one can describe the even Clifford group as $\Gamma^+(V,q)=\{v_1\cdots v_{2r}\in C(V,q)\mid v_i\in V, \prod_{i=1}^{2r}q(v_i)\neq 0\}.$
The algebra $C(V,q)$ has a canonical involution $\tau$ defined by $\tau(v_1\cdots v_r)=v_r\cdots v_1$ which restricts to $C_0(V,q)$ and is denoted as $\tau_0$.
We have a map $N\colon \Gamma^+(V,q)\rightarrow F^*$, called the {\bf norm} map, defined by $N(u)=\tau_0(u)u$, i.e., if $u=v_1\cdots v_{2r}$ then $N(u)=\prod_{i=1}^{2r}q(v_i)$. 
The kernel of $N$ is called the {\bf Spin} group and denoted as $Spin(V,q)$, i.e.,  we have 
$$1\rightarrow \{\pm 1\} \rightarrow Spin(V,q)\rightarrow O'(V,q)\rightarrow 1$$
where $O'(V,q)$ is {\bf reduced orthogonal group}.
In the case $F$ is algebraically closed field we have $F^*=(F^*)^2$ and $O'(V,q)=SO(q)$.
The group $Spin$ is a simply connected and semisimple algebraic group defined over $F$ and is a double cover of $SO(q)$.
And the group $\Gamma^+$ is a connected reductive algebraic group defined over $F$.
We write all of the above relations in a diagram below.
\[\xymatrix@R=1pc{& 1\ar[d] & 1\ar[d] & 1\ar[d] &\\
1\ar[r] & \{\pm 1\}\ar[r]\ar[d]& Spin(V,q)\ar[r]\ar[d] & O'(V,q)\ar[r]\ar[d] & 1\\
1\ar[r] & F^* \ar[r]\ar[d]& \Gamma^+(V,q)\ar[r]^{\chi}\ar[d]^{N} & SO(V,q)\ar[r]\ar[d] & 1\\
1\ar[r] & (F^*)^2\ar[r]& F^*\ar[r] & F^*/(F^*)^2\ar[r] & 1\\
}
\]

Here we would like to study which semisimple elements are real in the group $Spin(V,q)$ and what is there structure.
Structure of real elements in $SO(q)$ has been studied in \cite{st2} (see Section 3.4) but same question has not been addressed for $O'(V,q)$ in the literature.

\section{A Maximal Torus in $\Gamma^+(V,q)$}\label{torus}

A semisimple element in a connected algebraic group belongs to a maximal torus. 
Hence we would like to give a description of a maximal torus  in $\Gamma^+(V,q)$ which suits our need.
In this section we consider a quadratic form $q$ of maximal Witt index over a field $F$ and describe a maximal torus in $\Gamma^+(V,q)$ defined over $F$ with respect to the fixed Witt basis.
By abuse of notation, we would refer this maximal torus as standard maximal torus.
We choose a basis of $V$ as follows.
If $\dim(V)=2m$ we choose the basis $\mathcal B=\{e_1,f_1,\ldots,e_m,f_m\}$ and if $\dim(V)=2m+1$ we take $\mathcal B=\{e_0,e_1,f_1,\ldots,e_m,f_m\}$ where $q(e_0)\neq 0, B(e_i,f_j)=\delta_{ij}, B(e_i,e_j)=0=B(f_i,f_j)\ \forall i,j$.
A maximal torus of $SO(q)$ with respect to this basis is given by $\{\diag(\lambda_1,\lambda_1^{-1},\ldots,\lambda_m,\lambda_m^{-1}) \mid \lambda_i\in F^*\}$ or $\{\diag(1,\lambda_1,\lambda_1^{-1},\ldots,\lambda_m,\lambda_m^{-1})\mid \lambda_i\in F^*\}$ depending on $\dim(V)$ is even or odd.
We wish to describe a maximal torus in the group $\Gamma^+(V,q)$ and $Spin(V,q)$ which maps to this torus in $SO(q)$ under $\chi$.
We begin with a simple computation first.   
\begin{lemma}
With the notation as above, the element $\prod_{i=1}^m(e_i+f_i)(e_i+\lambda_i f_i)\in \Gamma^+(V,q)$ maps to $\diag(\lambda_1,\lambda_1^{-1},\ldots,\lambda_m,\lambda_m^{-1})$ in $SO(q)$ if $\dim(V)$ is even and maps to the element $\diag(1,\lambda_1,\lambda_1^{-1},\ldots,\lambda_m,\lambda_m^{-1})$ if $\dim(V)$ is odd.
\end{lemma}
\begin{proof}
First we claim that the element $(e_i+f_i)(e_i+\lambda_i f_i)\in \Gamma^+(V,q)$ maps to the element $\diag(1,\ldots,\lambda_i,\lambda_i^{-1},\ldots,1)$ in $SO(q)$ under the map $\chi$.
Since 
$$S_{e_i+\beta f_i}(e_i)=-\beta f_i, S_{e_i+\beta f_i}(f_i)=-\frac{1}{\beta}e_i,$$ the vector $e_i+\beta f_i\in \Gamma(V,q)$ maps to $-S_{e_i+\beta f_i}=\diag\left(-I,\ldots,\left(\begin{array}{cc}0&\frac{1}{\beta}\\ \beta&0 \end{array}\right),\ldots,-I\right)$ in $O(q)$.
Hence,  
$\chi((e_i+f_i)(e_i+\lambda_i f_i))$
\begin{eqnarray*}
&=&\diag\left(-I,\ldots,\left(\begin{array}{cc}0&1\\ 1&0 \end{array}\right),\ldots,-I\right) \diag\left(-I,\ldots,\left(\begin{array}{cc}0&\lambda_i^{-1}\\ \lambda_i&0 \end{array}\right),\ldots,-I\right)\\
&=&\diag\left(I,\ldots,\left(\begin{array}{cc}\lambda_i&0\\ 0&\lambda_i^{-1} \end{array}\right),\ldots,I\right). 
\end{eqnarray*}
Hence the required result follows.
\end{proof}
We use this Lemma to describe a maximal torus in $\Gamma^+(V,q)$.
\begin{lemma}\label{torusincliffordgroup}
The set $T=\left\{\lambda_0\prod_{i=1}^m (e_i+f_i)(e_i+\lambda_i f_i)\in \Gamma^+(V,q)\mid \lambda_i\in F^*\right\}$ is a maximal torus in $\Gamma^+(V,q)$ defined over $F$.
And the set $T'=T\cap ker(N)=\{\lambda_0\prod_{i=1}^m (e_i+f_i)(e_i+\lambda_i f_i) \in \Gamma^+(V,q)\mid \lambda_0^2\prod_{i=1}^m\lambda_i=1, \lambda_i\in F^*\}$ contains a maximal torus of $Spin(V,q)$.
\end{lemma}
\begin{proof}
We base change to $\bar F$ and define a map 
\begin{eqnarray*}
\phi\colon (\bar F^*)^{m+1} &\rightarrow& \bar T\subset \Gamma^+(\bar V,\bar q)\\
(\lambda_0,\lambda_1,\ldots,\lambda_m) &\mapsto& \lambda_0\prod_{i=1}^m (e_i+f_i)(e_i+\lambda_i f_i).
\end{eqnarray*} 
The map $\phi$ is an algebraic group homomorphism defined over $F$ since $(e_i+f_i)(e_i+\lambda f_i).(e_i+f_i)(e_i+\mu f_i)=\lambda\mu e_if_i+f_ie_i =(e_i+f_i)(e_i+\lambda\mu f_i)$ and the expression $(e_i+f_i)(e_i+\lambda f_i)$ in different $i$ commutes in $\Gamma^+(V,q)$ as they are orthogonal in $V$.
Hence $\bar T$ is a torus in $\Gamma^+(\bar V,\bar q)$ which is an inverse image of the (standard) maximal torus in $SO(\bar q)$ under the map $\chi$.
This implies that $T$ is a maximal torus in $\Gamma^+(V,q)$.
By looking at elements of norm $1$ in $T$ we get the required description of $T'$ in $Spin(V,q)$.
\end{proof}
\noindent We note that the set $S=\{\prod_{i=1}^m (e_i+\lambda_if_i)(e_i+\lambda_i^{-1} f_i)\in Spin(V,q)\mid \lambda_i\in F^*\}\}$ maps to $\{\diag(\lambda_1^{-2},\lambda_1^2,\ldots,\lambda_m^{-2},\lambda_m^2)\}$ also gives a description of a maximal torus in the $Spin$ group.
But this description is not useful for us.

\section{Involutions in $\Gamma^+(V,q)$}

Let $\tau$ be an involution in the orthogonal group $O(q)$.
Then there exists a nondegenerate subspace $W\subset V$ such that $\tau=-1|_{W}\oplus 1|_{W^{\perp}}$.
Hence the involutions in orthogonal groups are in one-one correspondence with the nondegenerate subspaces of $V$.
And an involution in $SO(q)$ corresponds to an even dimensional nondegenerate subspace of $V$.
We want to know when a nondegenerate subspace gives an involution in  $\Gamma^+(V,q)$.
\begin{proposition}
Let $\tau$ be an involution in $SO(V,q)$ which corresponds to a  nondegenerate subspace $W$ of dimension $2r$.
Then $\tau$ lifts to an involution in $\Gamma^+(V,q)$ if and only if $disc(W)=(-1)^r$ where $disc(W)=\prod_{i=1}^{2r} q(v_i)$ for some orthogonal basis $\{v_1,\ldots,v_{2r}\}$ of $W$.
\end{proposition}
\begin{proof}
To the element $\tau$ we associate a non degenerate subspace $W$ of dimension $2r$ and choose an orthogonal basis $\{v_1,\ldots,v_{2r}\}$ of $W$.
We see that the element $u=v_1\cdots v_{2r}\in \Gamma^+(V,q)$ maps to $\tau$.
We claim that the element $u$ is an involution if and only if  $\prod_{i=1}^{2r}q(v_i)=(-1)^{r}$.
This follows from
 $$u^2=v_1\cdots v_{2r}v_1\cdots v_{2r}=(-1)^{\frac{2r(2r-1)}{2}}\prod_{i=1}^{2r}q(v_i)=(-1)^r\prod_{i=1}^{2r}q(v_i).$$
\end{proof}
In view of the Witt basis chosen in previous section the element $u=(e_1+f_1)\cdots (e_m+f_m)$ has $u^2=(-1)^{\frac{m(m-1)}{2}}$. 
In the case of odd dimension the element $u=e_0(e_1+f_1)\cdots (e_m+d^{-1}f_m)$  where $N(e_0)=d$ has $u^2=(-1)^{\frac{m(m+1)}{2}}$.

\section{Real Semisimple Elements from $SO(q)$ to $\Gamma^+(V,q)$}

Let $\tilde t\in SO(q)$ be a semisimple element and $t\in \Gamma^+(V,q)$ such that $\chi(t)=\tilde t$.
Suppose $\tilde t$ is real.
Then there exists $s\in \Gamma^+(V,q)$ such that $sts^{-1}=\alpha t^{-1}$ for some $\alpha\in F^*$.
If $F$ is an algebraically closed field, $\Gamma^+(V,q)\cong F^*.Spin(V,q)$ and hence one can always choose $s$ in the $Spin$ group.
In what follows we look for the conjugating element $s$ in $Spin$ group which conjugates $t$ to $N(t)t^{-1}$.  

We treat an element of $Spin$ group as an element of $\Gamma^+(V,q)\subset \Gamma(V,q)$ and often do the calculations there.
First we consider the form of maximal Witt index $m$ where $\dim(V)=2m$ or $2m+1$.
We fix a Witt basis of $V$ as described in Section~\ref{torus}, $\{e_1,f_1,\ldots,e_m,f_m\}$ in the case $\dim(V)$ even and $\{e_0,e_1,f_1,\ldots,e_m,f_m\}$ if $\dim(V)$ is odd.
Let $t\in \Gamma^+(V,q)$ be a semisimple element which belongs to the standard maximal torus $T$, i.e., the maximal torus described with respect to the fixed Witt basis (see Lemma~\ref{torusincliffordgroup}).
By the structure theory of quadratic forms such a basis always exists over algebraically closed field.  
Then we have,
\begin{lemma}\label{negativeconjugate}
Let $F$ be an algebraically closed field and $\dim(V)=2m$ or $2m+1$.
Let $t\in \Gamma^+(V,q)$ be a semisimple element.
Then,
\begin{enumerate} 
\item there exists $s\in\Gamma(V,q)$ with $s^2=(-1)^{\frac{m(m-1)}{2}}$ with $N(s)=1$ such that $sts^{-1}=N(t)t^{-1}$.
Moreover, $s\in\Gamma^+(V,q)$ if $m$ is even.
\item If $m$ is odd and suppose $-1$ is an eigen value of $\chi(t)$ then there exists $s\in \Gamma^+(V,q)$ with $N(s)=1$ and $s^2=(-1)^{\frac{(m-1)(m-2)}{2}}$ such that $sts^{-1}=-N(t)t^{-1}$.
\end{enumerate}
\end{lemma}
\begin{proof} Since $F$ is algebraically closed, with out loss of generality we may assume $t$ belongs to the standard maximal torus $T$.
Let $t$ maps to $\diag(\alpha_1,\alpha_1^{-1},\ldots,\alpha_m,\alpha_m^{-1})$ or $\diag(1,\alpha_1,\alpha_1^{-1},\ldots,\alpha_m,\alpha_m^{-1})$ in $SO(q)$ depending on $\dim(V)$ even or odd.
Recall from Lemma~\ref{torusincliffordgroup} we can write $t=\beta \prod_{i=1}^m(e_i+f_i)(e_i+\alpha_if_i)$.
We see that $N(t)=\beta^2\prod_{i=1}^m\alpha_i$ and $t^{-1}= \beta^{-1}\left(\prod_{i=1}^m\alpha_i^{-1}\right)\prod_{i=1}^m(e_i+\alpha_if_i)(e_i+f_i)$.
Let $s=\prod_{i=1}^m(e_i+f_i)$.
Then,
 
\begin{eqnarray*}
sts^{-1} &= & \prod_{i=1}^m(e_i+f_i). \beta \prod_{i=1}^m(e_i+f_i)(e_i+\alpha_if_i). \prod_{i=m}^1(e_i+f_i)\\
&=&\beta \prod_{i=1}^m(e_i+\alpha_if_i)(e_i+f_i)\\
&=& \beta^2.\left(\prod_{i=1}^m\alpha_i\right) t^{-1}\\
&=& N(t)t^{-1}.
\end{eqnarray*}
Also we check that $s^2= (-1)^{\frac{m(m-1)}{2}}$ and $N(s)=1$.

For the proof of second part we write $t=\beta (e_1+f_1)(e_1-f_1)\prod_{i=2}^m(e_i+f_i)(e_i+\alpha_if_i)$.
In this case we take $s=\prod_{i=2}^m(e_i+f_i)$ and notice that $(e_1+f_1)(e_1-f_1) = - (e_1-f_1)(e_1+f_1)= ((e_1+f_1)(e_1-f_1))^{-1} = -N((e_1+f_1)(e_1-f_1))((e_1+f_1)(e_1-f_1))^{-1} $.
Rest of the proof is similar to the calculations done above.
\end{proof}

\begin{lemma}
Let $(V,q)$ be a vector space which is direct sum of $(V_1,q_1)$ and $(V_2,q_2)$.
Let $t_i\in \Gamma^+(V_i,q_i)$.
Suppose there exist $s_i\in \Gamma(V_i,q_i)$ such that $s_it_is_i^{-1}=N(t_i)t_i^{-1}$.
Suppose $t=t_1t_2\in \Gamma^+(V,q)$ and $s=s_1s_2$. 
Then $sts^{-1}=N(t)t^{-1}$.
\end{lemma}
\begin{proof}
Notice that for $x\in \Gamma^+(V_1,q_1)$ and $y\in \Gamma(V_2,q_2)$, $xy=yx$.
\end{proof}
\begin{theorem}\label{conjugatecliffordgp}
Let dimension of $V$ be $2m$ or $2m+1$.
Then, any semisimple element $t\in \Gamma^+(V,q)$ is conjugate to $N(t)t^{-1}$ by an element $s\in \Gamma(V,q)$ such that $N(s)=1$ and $s^2=(-1)^{\frac{m(m-1)}{2}}$.
Moreover, the element $s\in \Gamma^+(V,q)$ if $m$ is even.
\end{theorem}
\begin{proof}
Let us first assume $F=\bar F$, i.e., $F$ is an algebraically closed field.
Then the form $q$ has maximal Witt index and any semisimple element can be conjugated to an element in the standard torus $T$ with respect to a Witt basis. 
In this case the proof of the statement follows from Lemma~\ref{negativeconjugate}.
We use Galois theory argument, which is similar to~\cite{gp} Proposition 2.3, to deduce the general result over perfect fields.

Let $t\in \Gamma^+(V,q)$ be a semisimple element. 
We consider $\bar t\in\Gamma^+(\bar V,\bar q)$ where $\bar V=V\otimes \bar F$ and $\bar F$ denotes the algebraic closure of $F$.
We decompose $\bar V$ with respect to $\chi(\bar t)$ as follows:
$$\bar V=\bar V_1\oplus \bar V_{-1}\oplus_{\lambda\in \bar F^*}\bar W_{\lambda}$$ 
where $\bar W_{\lambda}=\bar V_{\lambda}\oplus \bar V_{\lambda^{-1}}$.
We notice that each of the subspace in this decomposition, except possibly $\bar V_1$, is of even dimension because $\chi(\bar t)\in SO(\bar q)$. 
We want to write the vector representation of $\bar t$ with respect to this decomposition. 
Let us denote the restriction of $\bar t$ to $\bar W_{\lambda}$ and $\bar V_{-1}$ by $\bar t_{\lambda}$ and $\bar t_{-1}$ respectively.
These representatives belong to the corresponding even Clifford groups, i.e., $\bar t_{\lambda}\in \Gamma^+(\bar W_{\lambda})$ and $\bar t_{-1}\in \Gamma^+(\bar V_{-1})$, as the dimensions of the respective subspaces are even.
Then $\bar t=\beta.\bar t_{-1}\prod_{\lambda}\bar t_{\lambda} $.
From Lemma~\ref{negativeconjugate} we can get $\bar s_{\lambda}$ with $N(\bar s_{\lambda})=1$ such that $\bar s_{\lambda}\bar t_{\lambda} \bar s_{\lambda}^{-1}= N(\bar t_{\lambda}) \bar t_{\lambda}^{-1}$ on each subspace corresponding to $\lambda$ and $-1$.

Let $\Lambda$ be the Galois group of $\bar F$ over $F$.
The subspace $W_{\Lambda\lambda}=\oplus_{\sigma\in\Lambda}\bar W_{\sigma\lambda}$ is defined over $F$ and the restriction of $\bar t$ to this subspace is $t_{\Lambda\lambda}=\oplus_{\sigma\in\Lambda}\bar t_{\sigma\lambda}$ where $\bar t_{\sigma\lambda}=\bar t|_{\bar W_{\sigma\lambda}}$.
Also $s_{\Lambda\lambda}=\prod_{\sigma\in\Lambda} \bar s_{\sigma\lambda}$ is defined over $F$ and hence $s=\bar s_{-1}\prod \bar s_{\lambda}=\bar s_{-1}\prod_{\lambda} s_{\Lambda\lambda}$ is defined over $F$ since $\bar V_{-1}$ is defined over $F$ and so is $\bar s_{-1}$.
Moreover $sts^{-1}=N(t)t^{-1}$ with $N(s)=1$ and length of $s$ in vector representation is $m$.
Hence the theorem follows.
\end{proof}

\begin{corollary}\label{real01spin}
Let $\dim(V)\equiv 0, 1 \imod 4$.
Then every semisimple element in $Spin(V,q)$ is real in $Spin(V,q)$.
\end{corollary}
\begin{remark}\label{minuseigenvalue}
With the notation in Theorem, suppose $m$ is odd.
In addition suppose $\chi(t)$ has eigen value $-1$ and hence the subspace $\bar V_{-1}$ is defined over $F$ which is even dimensional.
Then there exists $s\in \Gamma^+(V,q)$ with $N(s)=1$ such that $sts^{-1}=-N(t)t^{-1}$.
The proof of this follows from part $(2)$ of Lemma~\ref{negativeconjugate} with appropriate changes made in the proof of the Theorem.
\end{remark}

Notice that in the Theorem~\ref{conjugatecliffordgp}, the conjugating element $s$ need not belong to $\Gamma^+(V,q)$ always.
Hence we need to analyse the case when $m$ is odd.
Let us first consider $\dim(V)=2m$ and $m$ odd.
We recall a Theorem from~\cite{st2} (see the proof of Theorem 3.4.6 there),
\begin{theorem}\label{2mod4real}
If $\dim(V)=2m$ and $m$ odd, then a semisimple element $t\in SO(q)$ is real if and only if $1$ or $-1$ is an eigenvalue of $t$.  
\end{theorem}
\noindent There are semisimple elements (e.g. strongly regular elements) in $SO(q)$ which are not real.
Hence its lift in $Spin(V,q)$ is also not real.

Now we consider the case when $\dim(V)=2m+1$ where $m$ is odd.
\begin{lemma}\label{oddcase}
Suppose $t\in \Gamma^+(V,q)$ is semisimple.
We have $\dim(V_1)\geq 1$.
Then, $t$ is conjugate to $N(t)t^{-1}$ by an element $s$ in $\Gamma^+(V,q)$ such that $N(s)=1$ if and only if $N(V_1)\subset N(\mathcal Z_{\Gamma^+(V,q)}(t))$.
\end{lemma}
\begin{proof}
The element $t$ maps to $\diag(1,\lambda_1,\lambda_1^{-1},\ldots,\lambda_m,\lambda_m^{-1})$.
From the theorem above we get $s_1\in \Gamma(V,q)$ of length $m$ and $N(s_1)=1$ such that $s_1ts_1^{-1}=N(t)t^{-1}$.
We take $e_0\in V_1$ and consider $s=e_0s_1\in \Gamma^+(V,q)$.
Then $sts^{-1}=N(t)t^{-1}$ and $N(s)=N(e_0)=d$, say.
Suppose $\tilde s$ is another element with same property such that $N(\tilde s)=1$.
Then $\tilde s\in s\mathcal Z_{\Gamma^+}(t)$, say $\tilde s=sy$.
This implies $N(y)=d^{-1}$, i.e., $N(V_1)\subset N(\mathcal Z_{\Gamma^+(V,q)}(t))$.

Conversely, let $N(V_1)\subset N(\mathcal Z_{\Gamma^+(V,q)}(t))$.
Then there exists $e_0\in V_1$ and $s=e_0s_1$, as above, conjugates $t$ to $N(t)t^{-1}$.
Let $N(e_0)^{-1}=N(y)$ for some $y\in \mathcal Z_{\Gamma^+(V,q)}(t)$.
We consider $\tilde s=sy$ then $\tilde s$ has all required properties.
\end{proof}

\begin{theorem}\label{splittorus}
Let $n=2m+1$ where $m$ is odd.
Let $t$ be a semisimple element such that it has an eigenspace defined over $F$ of dimension more than $1$.
Then $t$ is conjugate to $N(t)t^{-1}$ by an element of $Spin(V,q)$. 
\end{theorem}
\begin{proof}
Since it has an eigenspace defined over $F$ of dimension more than $1$ the element $t$ looks like $(e_1+f_1)(e_1+\alpha_1f_1)\cdots (e_r+f_r)(e_r+\alpha_rf_r)\tilde t$ where $\alpha_i\in F$.
In this case from Theorem~\ref{conjugatecliffordgp}, we get element $s=(e_1+f_1)\cdots(e_r+f_r) \tilde s\in \Gamma(V,q)$ such that $sts^{-1}=N(t)t^{-1}$ with $N(s)=1$ and $s^2=\pm 1$.
We modify the element $s$ by a centraliser element and consider $s_1=e_0(e_1+d^{-1}f_1)(e_2+f_2)\cdots (e_r+f_r)\tilde s\in \Gamma^+(V,q)$ where $d=N(e_0)$.
Then $s_1\in Spin(V,q)$ and $s_1ts_1^{-1}=N(t)t^{-1}$ with $s^2=\pm 1$.
\end{proof}
\begin{corollary}
Let $n=2m+1$ with $m$ odd.
Let $t\in \Gamma^+(V,q)$ be an element of the standard torus $T$.
Then there exists $s\in\Gamma^+(V,q)$ such that $sts^{-1}=N(t)t^{-1}$ and $N(s)=1$, $s^2=(-1)^{\frac{m(m+1)}{2}}$.
\end{corollary}
\begin{remark}\label{nonrealspin3} 
We can define a map from $SU(2)$ to $SO(3)$ by its action on its Lie algebra.
In $SU(2)$ we have semisimple elements which are not real (ref. to \cite{st2}, Remark 1 following Theorem 3.6.2).
But every semisimple element is real in $SO(3)$.
This gives an example that a real semisimple element in $SO(3)$ need not lift to a real element in the $Spin$ group.
\end{remark}

\section{Real Elements in $Spin(V,q)$}

In this section we classify real semisimple elements in $Spin(V,q)$.
We make a simple observation first.
\begin{proposition}
Let $G$ be a linear group, i.e., a subgroup of $GL_n$.
Let $t\in G$ be real.
Then there exists $s\in G$ with $s^2=\pm 1$ such that $sts^{-1}=t^{-1}$ if and only if $t=\tau_1\tau_2$ with $\tau_1^2=\pm 1=\tau_2^2$. 
\end{proposition}
\noindent In the group $GL_n, SL_n, O_n(q), SO_n(q), Sp_n$ and groups of type $G_2$ every semisimple element can be decomposed in this way (Theorem~\ref{st12}).

We first collect the well known results over algebraically closed fields and finite fields.
When $F$ is algebraically closed field we have the exact sequence 
$$1 \rightarrow \{\pm 1\}\rightarrow Spin(V,q)\rightarrow SO(V,q)\rightarrow 1.$$
In view of this we have,
\begin{proposition}\label{realalgclosed}
Let $F$ be an algebraically closed field and $t\in Spin(V,q)$.
If $\dim(V)\equiv 0,1,3 \imod 4$ then $t$ has decomposition $t=\tau_1\tau_2$ where $\tau_1, \tau_2\in Spin(V,q)$ and $\tau_1^2=\pm 1=\tau_2^2$.
And if $t\in Spin(V,q)$ is a real semisimple element then it always has this decomposition.  
\end{proposition}
\begin{proof}
We recall that in the case of $\dim(V)\equiv 0,1,3 \imod 4$ every element of $SO(V,q)$ (and real semisimple element for all dimension) is a product of two involutions which combined with the above exact sequence implies the proposition (ref. \cite{kn}, Theorem A and \cite{st2} Theorem 3.4.6).
\end{proof}
\begin{proposition}\label{realfinitefield}
Let $F$ be a finite field and suppose $\dim(V)\equiv 0,1,3 \imod 4$.
Then semisimple elements in $Spin(V,q)$ are real.
\end{proposition}
\begin{proof}
This follows from \cite{tz} Proposition 3.1 (ii) by noting that $-1$ belongs to the Weyl group.
\end{proof}

Now we consider $F$ any perfect field of characteristic $\neq 2$.
\begin{theorem}\label{realspin}
Suppose $\dim(V)=0,1,2\imod 4$.
Let $t\in Spin(V,q)$ be a semisimple element.
Then $t$ is real in $Spin(V,q)$ if and only if $t=\tau_1\tau_2$ where $\tau_1^2=\pm 1=\tau_2^2$ and $\tau_1,\tau_2\in Spin(V,q)$.
\end{theorem}
\begin{lemma}
Let $\dim(V)=2m$ and $m$ odd.
Then a semisimple element $t\in Spin(V,q)$ is real in $Spin(V,q)$ if and only if $1$ is an eigenvalue of $\chi(t)$.
\end{lemma}
\begin{proof}
Suppose $1$ is an eigenvalue.
Since $\dim(V)$ is even, $dim(V_1)$ is also even.
In this case the proof of Theorem~\ref{conjugatecliffordgp} produces the required element.

Now suppose $t$ is real.
Then $\chi(t)$ is real in $SO(q)$.
Which is if and only if $1$ or $-1$ is an eigenvalue of $\chi(t)$.
Suppose $1$ is not an eigenvalue of $\chi(t)$.
Since $-1$ is an eigenvalue the eigenspace $V_{-1}$ is defined over $F$.
Then from Remark~\ref{minuseigenvalue} we have an element $s\in Spin(V,q)$ such that $sts^{-1}=-t^{-1}$.
Since $t$ is real, we get $t$ is conjugate to $-t$ in $Spin(V,q)$.
But from Theorem 2A, in \cite{fs},  $t$ is conjugate to $-t$ in $Spin(V,q)$ if and only if $1$ and $-1$ both are eigenvalues of $\chi(t)$, a contradiction.
\end{proof}
\noindent Notice that, in this case, a semisimple element in $SO(q)$ is real (see Theorem~\ref{2mod4real}) if and only if $1$ or $-1$ is an eigen value.
But only those real elements which have eigen value $1$ can be lifted to a real element in $Spin(V,q)$.

\begin{proof}[Proof of Theorem~\ref{realspin}]
If $\dim(V)=0,1\imod 4$ then every semisimple element has this property (see Corollary~\ref{real01spin}).
If $\dim(V)=2\imod 4$ the result follows from the Lemma above.
\end{proof}

\begin{corollary}\label{counterexample}
With notation as above, if $\dim(V)\equiv 0,1,2\imod 8$ then any real semisimple element $t\in Spin(V,q)$ can be written as $t=\tau_1\tau_2$ where $\tau_1^2=1=\tau_2^2$ and if $\dim(V)\equiv 4,5,6\imod 8$ then $t$ can be written as $t=\tau_1\tau_2$ where $\tau_1^2=\pm 1=\tau_2^2$.
\end{corollary}
\noindent The group $Spin(4,\mathbb C)\cong SL_2(\mathbb C)\times SL_2(\mathbb C)$ and $Spin(5,\mathbb C)\cong Sp_2(\mathbb C)$ and in both of these groups there are semisimple elements which are not product of two involutions but have decomposition $t=\tau_1\tau_2$ where $\tau_i^2=-1$ (see~\cite{st2} Lemma 3.2.1 and Lemma 3.5.2). 
In the case $\dim(V)\equiv 2 \imod 4$ we have $Spin(6,\mathbb C)\cong SL_4(\mathbb C)$. 
We consider the covering map $\rho\colon SL_4(\mathbb C)\rightarrow SO(\Lambda^2\mathbb C^4,B)\cong SO(6)$ defined in~\cite{gw} exercise 5 in section 2.3.5.
With respect to appropriate basis of  $SO(\Lambda^2\mathbb C^4,B)$, say, $\{e_1\wedge e_2,e_3\wedge e_4, -e_1\wedge e_3, e_2\wedge e_4, e_1\wedge e_4, e_2\wedge e_3\}$ we have the diagonal torus in $SO(\Lambda^2\mathbb C^4,B)$.
The semisimple element $\diag(\mu_1\mu_2\mu_3,\mu_1\mu_2^{-1}\mu_3^{-1},\mu_1^{-1}\mu_2\mu_3^{-1}, \mu_1^{-1}\mu_2^{-1}\mu_3)$ in $SL_4(\mathbb C)$ maps to $\diag(\mu_1^2,\mu_1^{-2}, \mu_2^2, \mu_2^{-2}, \mu_3^2, \mu_3^{-2} )$.
We see that the element $\diag(1,1,a^2,a^{-2},b^2,b^{-2})$ in $SO(6)$ lifts to $\diag(ab,a^{-1}b^{-1},ab^{-1},a^{-1}b)$ in $SL_4(\mathbb C)$ which is real and in fact product of two involutions.
The element $\diag(-1,-1,-1,-1,a^2,a^{-2})$ for $a^2\neq 1$ in $SO(6)$ lifts to $\diag(-a,a^{-1},a^{-1},-a)$ in $SL_4(\mathbb C)$ which is not real.

Now we consider the case when $\dim(V)=2m+1$ where $m$ is odd.
In this case we know that all semisimple elements are real over algebraically closed field and finite field (Proposition~\ref{realalgclosed} and \ref{realfinitefield}) but not necessarily over arbitrary $F$ (see remark~\ref{nonrealspin3}).
However we have,
\begin{proposition}
Let $t$ be either a strongly regular element or a semisimple element in  a split torus. 
Then $t$ is real if and only if it has decomposition $t=\tau_1\tau_2$ with $\tau_i^2=\pm 1$.
\end{proposition}
\begin{proof}
Notice that the centraliser in this case, is the maximal torus containing it.
Hence the element $\tilde s$ in Lemma~\ref{oddcase} satisfies $\tilde s^2=\pm 1$.
If $t$ belongs to a split torus then it has an  eigen value over $F$ and the result follows by Theorem~\ref{splittorus}. 
\end{proof}
\noindent However, it is not clear what is the structure of a real semisimple element over $F$ in this case.


\end{document}